\documentclass[reqno,10pt,centertags]{amsart}
\usepackage{amssymb}
\usepackage{hyperref}
\usepackage{latexsym}
\usepackage{amsmath}
\usepackage[usenames,dvipsnames,svgnames,table]{xcolor}

\allowdisplaybreaks

\newcommand*{\mailto}[1]{\href{mailto:#1}{\nolinkurl{#1}}}

   \def\sH{{\mathfrak H}}

      \def\dC{{\mathbb C}}

      \def\dR{{\mathbb R}}

   \def\cB{{\mathcal B}}

\def\R{\mathbb{R}}

\def\N{\mathbb{N}}

\def\ran{{\text{\rm ran\,}}}

\def\dom{{\text{\rm dom\,}}}

\renewcommand{\dot}{\overset{\textbf{\Large.}}}

\newcommand{\no}{\notag}
\newcommand{\lb}{\label}
\newcommand{\f}{\frac}

\newcommand{\Oh}{O}

\newcommand{\dott}{\,\cdot\,}

\newtheorem{theorem}{Theorem}[section]
\newtheorem*{thm*}{Theorem}

\newtheorem{corollary}[theorem]{Corollary}

\theoremstyle{definition}

\newtheorem{hypothesis}[theorem]{Hypothesis}

\newtheorem{remark}[theorem]{Remark}

\numberwithin{equation}{section}

\begin{document}
\title{Weak Coupling and Spectral Instability for Neumann Laplacians}

\author[J. Behrndt]{Jussi Behrndt}
 \address{Technische Universit\"{a}t Graz\\
 Institut f\"ur Angewandte Mathematik\\
 Steyrergasse 30\\
 8010 Graz, Austria}

\email{behrndt@tugraz.at}
\urladdr{https://www.math.tugraz.at/~behrndt/}

\author[F.\ Gesztesy]{Fritz Gesztesy}
\address{Department of Mathematics,
Baylor University, Sid Richardson Bldg., 1410 S.\,4th Street,
Waco, TX 76706, USA}

\email{Fritz\_Gesztesy@baylor.edu}
\urladdr{https://math.artsandsciences.baylor.edu/person/fritz-gesztesy-phd}

\author[H. de Snoo]{Henk de Snoo}
\address{Bernoulli Institute for Mathematics, Computer Science and Artificial Intelligence,
University of Groningen, P.O. Box 407, 9700 AK Groningen, Netherlands}
\email{\mailto{h.s.v.de.snoo@rug.nl}}
\email{hsvdesnoo@gmail.com}

\date{\today}


\begin{abstract} 
We prove an abstract criterion on spectral instability of nonnegative selfadjoint extensions of a symmetric operator and apply this to self-adjoint Neumann Laplacians on bounded Lipschitz domains, intervals, and graphs. Our results can be viewed as variants of the classical weak coupling phenomenon for Schr\"odinger operators in $L^2(\dR^n)$ for $n=1,2$.
\end{abstract}

\maketitle

\section{Introduction} \lb{s1}
We start by recalling
the classical weak coupling phenomenon for Schr\"odinger operators, which goes back to Simon \cite{Si76}, \cite{Si77}. For this purpose, let $-\Delta$ be the self-adjoint one-dimensional Laplacian in $L^2(\dR)$ defined on $H^2(\dR)$
and assume that the potential $V:\dR\rightarrow\dR$ satisfies $V \in L^1(\dR;(1+x^2)dx)\cap L^2(\dR;dx)$, and $V$ is not zero a.e. For $\alpha\in\dR$ it follows that 
$-\Delta+\alpha V$ is self-adjoint in $L^2(\dR)$ and 
\begin{equation}
\sigma_{ess}(-\Delta+\alpha V)=\sigma_{ess}(-\Delta)=\sigma(-\Delta)=[0,\infty). 
\end{equation} 
It was shown in \cite[Theorem 2.5]{Si76} (see also \cite[Theorem XIII.11]{RS78}) that for any $\alpha<0$
one has 
 \begin{equation}
 \sigma_p(-\Delta+\alpha V)\cap (-\infty,0)\not=\emptyset\quad\text{if and only if}\quad \int_\dR V(x)\,dx\geq 0,
 \end{equation}
 and hence, in particular, if $V\geq 0$, then $\sigma_p(-\Delta+\alpha V)\cap (-\infty,0)\not=\emptyset$ for any $\alpha<0$. 
 The same result holds also for the self-adjoint Laplacian $-\Delta$ in $L^2(\dR^2)$
 under slightly different integrability conditions on the potential $V:\dR^2\rightarrow\dR$, and it is also
 well known that the phenomenon of weakly coupled eigenvalues does not appear in dimensions $n\geq 3$. The works \cite{Si76}, \cite{Si77} by Simon have inspired and influenced a lot of future research; they were followed by Klaus and Simon \cite{KS80}, \cite{KS80a}, and Rauch \cite{Ra80}. A wealth of additional information can be found, for instance, in \cite{BGS77}, \cite{BGRS97}, \cite{FK98}, \cite{GH87}, \cite{Ho85}, \cite{Kl77}--\cite{Kl82}, \cite{La80}, \cite{Me02}, \cite{Pa80}, \cite{Pa82}, \cite[Theorem~XIII.11, p.~336--338]{RS78}. For some other related more recent developments we refer the reader to \cite{B24}, \cite{CM21}, \cite{CS18}, 
 \cite{EKL18},  \cite{EKL24}, \cite{FMV11}, \cite{HHRV23}, \cite{KL14}, \cite{MV23}, and the references cited therein.
 
The main objective of this note is to transfer these ideas from Schr\"odinger operators $- \Delta + \alpha V$ to an abstract setting that replaces 
the Laplacian by a nonnegative self-adjoint extension $A$ of a densely defined closed nonnegative symmetric operator $S$ in a Hilbert space $\sH$ and the potential by an appropriate 
nonnegative self-adjoint perturbation, also denoted by $V$, that is relatively form compact with respect to $A$. 
In our main abstract result Theorem~\ref{fieldsmedalsk} it is shown that under some additional mild assumptions $A$ is spectrally unstable, that is, for any $\alpha<0$
the perturbed self-adjoint operator $A+\alpha V$ has negative discrete eigenvalues. The proof of Theorem~\ref{fieldsmedalsk} is based on the Birman--Schwinger principle, 
see, for instance, \cite{GLMZ05,KK66}. In fact, the essential assumptions to ensure the existence of weakly coupled negative eigenvalues of $A+\alpha V$, $\alpha<0$,
are $\ker(A)\not=\{0\}$ and $\ker(A)\not\subseteq \ker(V)$; roughly speaking the first assumption $\ker(A)\not=\{0\}$ ensures 
that the resolvent of $A$ has a singularity at $0$ and the second assumption $\ker(A)\not\subseteq \ker(V)$ is needed to preserve this singularity for the sandwiched resolvent 
$V^{1/2}(A-\mu I_{\sH})^{-1}V^{1/2}$ when $\mu<0$ tends to $0$. We note that for the special case where  $0$ is an isolated eigenvalue of finite multiplicity of $A$, our result would also follow from asymptotic perturbation theory.

Our general result applies directly to the Neumann Laplacian $-\Delta_N$ on a bounded interval $(a,b)$ or on a bounded Lipschitz domain $\Omega\subset\dR^n$, $n\geq 2$, since in that case 
$\ker(-\Delta_N)$ is spanned by the constant function and if $V\geq 0$ is a multiplication operator which is relatively compact perturbation with respect to $-\Delta_N$, then certainly 
$\ker(-\Delta_N)\not\subseteq\ker(V)$ as otherwise $V=0$ a.e. on $(a,b)$  or $\Omega$, respectively. As a consequence of our abstract result, Theorem~\ref{fieldsmedalsk},  
we conclude in Corollary~\ref{cor1} and Corollary~\ref{cor2} that for any $\alpha<0$ and 
nonnegative function $V$, $V\not=0$, such that
$V\in L^p$ with $p\geq 2$ if $n=1,2$ and $p>2n/3$ if $n\geq 3$, there exist weakly coupled negative bound states for the perturbed Neumann Laplacian $-\Delta_N+\alpha V$, that is,
\begin{equation}
 \sigma(-\Delta_N+\alpha V)\cap(-\infty,0)\not=\emptyset\quad\text{for any}\,\,\alpha<0. 
\end{equation}
We note that weakly coupled bound states for Schr\"odinger operators in $\dR^n$ exist only for $n=1,2$, whereas weakly coupled bound states for the perturbed Neumann Laplacian
exist in \textit{any} space dimension. We mention that our abstract result also applies to other self-adjoint nonnegative realizations $A$ of the Laplacian on bounded domains 
with the property $\ker(A)\not=\{0\}$ (cf.\ Remark~\ref{saythis}). The observations for the case of a bounded interval extend naturally to finite compact graphs, where the Neumann Laplacian
corresponds to Kirchhoff or standard boundary conditions; see, Corollary~\ref{graphi}. Furthermore, in Corollary~\ref{pncor} we consider a Sturm--Liouville operator with Neumann boundary conditions in $L^2((0,\infty))$ with $0$ as embedded eigenvalue at the bottom of the essential spectrum.

Finally, a few remarks about the notation employed: Given a separable complex Hilbert space $\sH$, $(\dott,\dott)_{\sH}$ denotes the scalar product in $\sH$ (linear in the second factor), $\| \,\cdot\,\|_{\sH}$ the norm in $\sH$, and $I_{\sH}$ represents the identity operator on $\sH$. The domain and range of a linear operator $T$ in $\sH$ are abbreviated by $\dom(T)$ and $\ran(T)$. 
The kernel (null space) of $T$ is denoted by
$\ker(T)$. The spectrum, point spectrum (i.e., the set of eigenvalues), essential spectrum, and resolvent set of a self-adjoint operator in $\sH$ will be abbreviated by 
$\sigma(\,\cdot\,)$, $\sigma_{p}(\,\cdot\,)$, 
$\sigma_{ess}(\,\cdot\,)$, and $\rho(\,\cdot\,)$, respectively. 
The space of compact linear operators in $\sH$ is denoted by $\cB_\infty(\sH)$. 
For $\Omega \subseteq \R^n$, $n \in \N$, we will abbreviate $L^2(\Omega; d^nx)$ for simplicity by $L^2(\Omega)$, and $I_{L^2(\Omega)}$ for convenience by $I$.

\section{Spectral instability of nonnegative self-adjoint extensions} \lb{s2} 

Throughout this section suppose that $S$ is a densely defined closed symmetric operator in a Hilbert space $\sH$
and assume that $S$ is semibounded from below with the lower bound $\kappa\geq 0$, that is,
\begin{equation}
 (Sf,f)_{\sH}\geq \kappa (f,f)_{\sH},\quad f\in\dom (S).
\end{equation}

\begin{hypothesis} \lb{h1}
Let $A$ be a nonnegative self-adjoint extension of $S$ in $\sH$ such that 
$\ker A\not=\{0\}$ and let $V\geq 0$ be a self-adjoint operator in $\sH$ which is relatively compact with respect to $A$, that is, 
\begin{equation}
\dom (A) \subseteq \dom (V)\quad\text{and}\quad V(A+I_{\sH})^{-1}\in\cB_\infty(\sH).
\end{equation} 
\end{hypothesis}

One notes that the lower bound of $A$ is $\kappa \geq 0$ and recalls that the Friedrichs extension $A_F$ of $S$ has the same lower bound as $S$. In the case of differential operators (see the next section) 
the reader may think of $A$ in Hypothesis~\ref{h1} as the self-adjoint Laplacian with Neumann boundary conditions. Another typical example for a self-adjoint extension of $S$ satisfying Hypothesis~\ref{h1} is the 
Krein--von Neumann extension $A_K$, the smallest nonnegative extension of $S$, which in the case $\kappa>0$ is defined by
\begin{equation}
 A_K=S^*\upharpoonright \dom (A_K),\quad \dom (A_K)=\dom (S)\dot +\ker (S^*)  
\end{equation}
(see, e.g., \cite[Sect.~5.4]{BHS20}, \cite{FGKLNS21} and the references cited therein). We also note that the 
self-adjoint extension theory point of view is not strictly necessary for the following arguments and Theorem~\ref{fieldsmedalsk} below, however we find it useful to compare 
$A$ in Hypothesis~\ref{h1} with the extremal nonnegative self-adjoint extensions $A_F$ and $A_K$. We will return to this topic elsewhere.

Our goal is to show that the lower bound $0$ for $A$ in Hypothesis~\ref{h1}
is not stable
under arbitrary small negative perturbations $\alpha V$. The relative compactness assumption in Hypothesis~\ref{h1} ensures that the operators $A+\alpha V$, $\alpha\in\dR$, are self-adjoint and that
\begin{equation} 
\sigma_{ess}(A+\alpha V)=\sigma_{ess}(A),
\end{equation}  
see, for instance, \cite[Theorem XIII.14 and Corollary 2]{RS78}.
Furthermore, 
\begin{itemize}
 \item [$(i)$] If $\alpha\geq 0$, then $A+\alpha V\geq 0$ and, in particular, 
$\sigma(A+\alpha V)\cap(-\infty,0)=\emptyset$.\smallskip 
 \item [$(ii)$] If $\alpha< 0$, then $\sigma(A+\alpha V)\cap(-\infty,0)$ is either empty or consists of 
discrete eigenvalues.
\end{itemize}

From Hypothesis \ref{h1} one obtains 
$V(A-z I_{\sH})^{-1}\in\cB_\infty(\sH)$, $z \in \rho(A)$,
by using the resolvent identity. We also note that 
\begin{equation} 
V^{1/2}(A+ I_{\sH})^{-1/2}\in\cB_\infty(\sH) 
\end{equation}
by
\cite[Theorem 3.5\,$(i)$]{GMMN08}. Then one has $V^{1/2}(A+ I_{\sH})^{-1}\in\cB_\infty(\sH)$, 
\begin{equation}\label{twee}
V^{1/2}(A-z I_{\sH})^{-1}\in\cB_\infty(\sH),\quad\text{and}\quad
V^{1/2} (A-z I_{\sH})^{-1/2}\in\cB_\infty(\sH), \quad z \in \rho(A).
\end{equation}
It follows  that $(A-z I_{\sH})^{-1}V^{1/2}$ and $(A-z I_{\sH})^{-1/2}V^{1/2}$, $z\in\rho(A)$, are densely defined bounded operators, whose closures coincide with the adjoints of the operators in 
\eqref{twee} for $\bar z\in\rho(A)$, and hence also belong to $\cB_\infty(\sH)$.
Therefore, the Birman-Schwinger family $K(z)$, defined by 
\begin{equation}\label{birmi}
K(z):=\overline{V^{1/2}(A - z I_{\sH})^{-1}V^{1/2}},\quad z\in\rho(A),
\end{equation}
satisfies
\begin{equation}\label{birmi2}
 K(z)=V^{1/2}(A - z I_{\sH})^{-1/2}\overline{(A - z I_{\sH})^{-1/2}V^{1/2}} \in\cB_\infty(\sH),\quad z\in\rho(A).
 \end{equation}
Thus, if $z\in\rho(A)$ and
$\alpha^{-1}$ is not an eigenvalue of the compact operator $K(z)$, then $K(z)+\alpha^{-1} I_{\sH}$ is boundedly invertible and one verifies in the same way as in 
\cite[Proof of Theorem 2.3]{GLMZ05} that in the present case $z\in\rho(A+\alpha V)$ and the resolvent formula
\begin{align}\label{resi}
(A+\alpha V - z I_{\sH})^{-1} & =(A - z I_{\sH})^{-1}      \\
& \quad -\overline{(A - z I_{\sH})^{-1}V^{1/2}}\bigl[K(z)+\alpha^{-1} I_{\sH}\bigr]^{-1}V^{1/2} (A - z I_{\sH})^{-1},    \no \\
& \hspace*{5.5cm} z \in \rho(A+\alpha V) \cap \rho(A),    \no 
\end{align}
holds.

The next theorem is our main abstract result; it provides a sufficient condition for spectral instability of the self-adjoint operator $A$ in Hypothesis~\ref{h1}.

\begin{theorem}\label{fieldsmedalsk}
Let $A$ and $V$ be as in Hypothesis~\ref{h1} and
assume, in addition, that $\ker(A)\not\subseteq \ker(V)$. Then 
\begin{equation} 
\sigma(A+\alpha V)\cap(-\infty,0)\not=\emptyset \, \text{ for {\bf any} $\alpha<0$.} 
\end{equation} 
\end{theorem}
\begin{proof}
By assumption there exists $k\in\ker(A)$, $\Vert k\Vert_{\sH}=1$, such that $Vk\not=0$ and hence also $V^{1/2} k \not= 0$. From this we conclude that 
there exists $f\in\dom \big(V^{1/2}\big)$ such that $h= V^{1/2} f$ satisfies $(h,k)_{\sH} \not= 0$ as otherwise 
$k\in\big(\ran \big(V^{1/2}\big)\big)^\bot=\ker\big(V^{1/2}\big)$.
We shall now make use of the orthogonal direct sum 
 decomposition
 \begin{equation}\label{decospacesk}
 \sH=\textrm{lin.span}\{k\}\,\oplus\,\bigl(\textrm{lin.span}\{k\}\bigr)^\bot 
\end{equation}
and denote the orthogonal projection in $\sH$ onto $(\textrm{lin.span}\{k\})^\bot$ by $P$. Then 
\begin{equation}\label{h0}
V^{1/2} f=h= (h, k)_{\sH} k + Ph
\end{equation}and for $\nu <0$ 
 it follows from $(A-\nu I_{\sH})^{-1}k=-\frac{1}{\nu} k$ and \eqref{decospacesk}  that 
 \begin{equation*}
 \begin{split}
  \bigl(V^{1/2}(A-\nu I_{\sH})^{-1} V^{1/2}f,f\bigr)_{\sH}&= \bigl((A-\nu I_{\sH})^{-1} h,h\bigr)_{\sH}  \\
  &=\bigl((A-\nu I_{\sH})^{-1} ((h, k)_{\sH} k+Ph), (h, k)_{\sH} k+Ph\bigr)_{\sH}  \\ 
   &=- \frac{\vert(h, k)_{\sH}\vert^2}{\nu}(k,k)_{\sH}+\bigl((A-\nu I_{\sH})^{-1} Ph,P h\bigr)_{\sH}   \\
 &=- \frac{\vert(h, k)_{\sH}\vert^2}{\nu}+\int_0^\infty \frac{1}{\lambda-\nu}d (E_A(\lambda)Ph,Ph)_{\sH},
 \end{split}
 \end{equation*}
where $E_A(\lambda)$, $\lambda \in \dR$, denotes the family of spectral projections of the self-adjoint operator $A$.
Since $(h, k)_{\sH}\not=0$ the first 
term tends to $+\infty$ as $\nu\uparrow 0$ and by
monotone convergence the spectral integral converges in $[0,+\infty]$ as $\nu\uparrow 0$. Hence, we conclude
 \begin{equation}\label{gutesache}
\lim_{\nu\uparrow 0}\,\bigl(V^{1/2}(A-\nu I_{\sH})^{-1} V^{1/2}f,f\bigr)_{\sH}=+\infty.
\end{equation}

We note that for $\nu<0$ the Birman--Schwinger operator $K(\nu)$ in \eqref{birmi}--\eqref{birmi2} is nonnegative and compact. Furthermore,
from \eqref{gutesache} we conclude that
\begin{equation}
\lim_{\nu\uparrow 0}\,\bigl\Vert K(\nu) \bigr\Vert_{\cB(\sH)} = +\infty
\end{equation}
and since the operator norm of the nonnegative compact operator $K(\nu)$, $\nu <0$, coincides with its largest eigenvalue 
we conclude that for any $\alpha<0$ there exist $\nu_\alpha<0$ and $k_\alpha\in\sH$, $k_\alpha\not=0$, such that 
\begin{equation}\label{qaa}
 K(\nu_\alpha)k_\alpha=-\frac{1}{\alpha} k_\alpha. 
\end{equation}
Now consider $f_\alpha=\overline{(A-\nu_\alpha I_{\sH})^{-1} V^{1/2}} k_\alpha$ (see also \cite{KK66} or \cite[Proof of Theorem~3.2]{GLMZ05} for the following arguments) 
and observe first that
\begin{equation}\label{nimmes}
\begin{split}
k_\alpha&= \bigl[K(z)+\alpha^{-1} I_{\sH}\bigr]^{-1}\bigl[K(z)-K(\nu_\alpha)\bigr]k_\alpha\\
&=(z-\nu_\alpha)\bigl[K(z)+\alpha^{-1} I_{\sH}\bigr]^{-1}V^{1/2}(A-z I_{\sH})^{-1}\overline{(A-\nu_\alpha I_{\sH})^{-1} V^{1/2}} k_\alpha\\
&=(z-\nu_\alpha)\bigl[K(z)+\alpha^{-1} I_{\sH}\bigr]^{-1}V^{1/2}(A-z I_{\sH})^{-1}f_\alpha, \quad z \in \rho(A), 
\end{split}
\end{equation}
and hence, in particular, $f_\alpha\not=0$ as otherwise $k_\alpha=0$. Using \eqref{nimmes} we see on the one hand
\begin{equation}\label{111}
 \begin{split}
  &\overline{(A-z I_{\sH})^{-1} V^{1/2}} k_\alpha\\
  &\quad=(z-\nu_\alpha)\overline{(A-z I_{\sH})^{-1} V^{1/2}}\bigl[K(z)+\alpha^{-1} I_{\sH}\bigr]^{-1}V^{1/2}(A-z I_{\sH})^{-1}f_\alpha\\
  &\quad= (z-\nu_\alpha)\bigl[(A-z I_{\sH})^{-1}-(A+\alpha V-z I_{\sH})^{-1}\bigr]f_\alpha,
 \end{split}
\end{equation}
where \eqref{resi} was used in the last equality. On the other hand, by the resolvent identity one obtains 
\begin{equation}\label{222}
 \begin{split}
  &\overline{(A-z I_{\sH})^{-1} V^{1/2}} k_\alpha\\
  &\quad = \overline{(A-\nu_\alpha I_{\sH})^{-1} V^{1/2}} k_\alpha + (z-\nu_\alpha)(A-z I_{\sH})^{-1}\overline{(A-\nu_\alpha I_{\sH})^{-1} V^{1/2}} k_\alpha\\
  &\quad = f_\alpha + (z-\nu_\alpha)(A-z I_{\sH})^{-1} f_\alpha.
 \end{split}
\end{equation}
It follows from \eqref{111} and \eqref{222} that  
$(\nu_\alpha-z) (A+\alpha V-z I_{\sH})^{-1}f_\alpha=f_\alpha$
which implies $f_\alpha\in\dom(A+\alpha V)$ and $(A+\alpha V) f_\alpha=\nu_\alpha  f_\alpha$. Hence $\nu_\alpha$ is an eigenvalue of $A+\alpha V$, thus
$\sigma(A+\alpha V)\cap(-\infty,0)\not=\emptyset$ for any $\alpha < 0$.
\end{proof}

\begin{remark}\label{remarkjussi}
We note that for the unperturbed nonnegative self-adjoint operator $A$ in Hypothesis~\ref{h1} it is only assumed that $0 \in \sigma_p(A)$, but no further restrictions on the spectrum of $A$ are required; for example, in general $0$ may be an eigenvalue of infinite multiplicity or an accumulation point of positive spectrum of $A$. In the special case where $0$ is an isolated eigenvalue of finite multiplicity of $A$, the spectral instability of $A$ in Theorem~\ref{fieldsmedalsk} would already follow from well-known results 
in analytic perturbation theory, see, for instance, \cite[Sect.~VII.3]{Ka80}, \cite[Theorems~XII.8, XII.9]{RS78}, \cite[Ch.~II]{Re69} and monotonicity of eigenvalues.
${}$ \hfill $\diamond$
\end{remark}

\section{Spectral instability of the Neumann Laplacian} \lb{s3} 

In this section we shall show that Theorem~\ref{fieldsmedalsk} applies to the Neumann Laplacian on bounded Lipschitz domains,  (arbitrary) intervals, and graphs, and conclude spectral instability for certain classes of potentials $V$ that are relatively compact.

In the following let $\Omega\subset\dR^n$, $n\geq 2$, be a bounded Lipschitz domain and let $\nu$ be the unit normal vector field pointing outwards on $\partial\Omega$. 
We shall use the notation
\begin{equation} 
H^{3/2}_\Delta(\Omega)=\bigl\{f\in H^{3/2}(\Omega) \, \big| \, \Delta f\in L^2(\Omega)\bigr\},
\end{equation} 
where $H^{3/2}(\Omega)$ is the $L^2$-based Sobolev space on $\Omega$ of fractional order $3/2$. We recall from \cite{BGM19,GM11}
that the Dirichlet trace mapping 
$C^\infty(\overline\Omega)\ni f\mapsto f\vert_{\partial\Omega}$  and 
the Neumann trace mapping 
$C^\infty(\overline\Omega)\ni f\mapsto \nu\cdot\nabla f\vert_{\partial\Omega}$ 
extend by continuity to continuous surjective mappings 
\begin{equation}\label{taus}
 \tau_D: H^{3/2}_\Delta(\Omega) \rightarrow H^{1}(\partial\Omega)\quad\text{and}\quad \tau_N: H^{3/2}_\Delta(\Omega)\rightarrow L^2(\partial\Omega),
\end{equation}
respectively,
where $H^{1}(\partial\Omega)$ denotes the first-order $L^2$-based Sobolev space on $\partial\Omega$. In the next corollary we study the weak coupling behaviour of the 
Neumann Laplacian 
\begin{equation}\label{neumannii}
 A_N f=-\Delta f,\quad f \in \dom (A_N)= \bigl\{g\in H^{3/2}_\Delta(\Omega)\, \big| \,\tau_N g=0\bigr\},
\end{equation}
which is self-adjoint in $L^2(\Omega)$, see, for instance, \cite[Theorem 6.10]{BGM19} or \cite[Theorem 2.6 and Lemma 4.8]{GM08} and also \cite{JK81}.

\begin{corollary}\label{cor1}
Let $\Omega\subset\dR^n$ be a bounded Lipschitz domain, $n \in \N$, $n\geq 2$, suppose that  $A_N$ is the self-adjoint Neumann Laplacian in $L^2(\Omega)$, and assume that $V\not=0$ is a nonnegative function such that $V\in L^p(\Omega)$ with $p\geq 2$ if $n=2$ and $p>2n/3$ if $n\geq 3$.
Then 
\begin{equation}
 (A_N + \alpha V)f=-\Delta f+\alpha Vf,\quad f \in \dom (A_N + \alpha V)=\dom (A_N),
\end{equation}
is self-adjoint in $L^2(\Omega)$, 
\begin{equation} 
V(A_N - z I)^{-1} \in \cB_{\infty}\big(L^2(\Omega)\big), \quad z \in \rho(A_N), 
\end{equation} 
and 
\begin{equation} 
\sigma(A_N+\alpha V)\cap(-\infty,0)\not=\emptyset \, \text{ for {\bf any} $\alpha<0$.}
\end{equation} 
Moreover, for $0 < -\alpha$ sufficiently small, the unique eigenvalue $\nu(\alpha) \in (-\infty,0)$ of $A_N + \alpha V$ satisfies 
\begin{equation}
\nu (\alpha) \underset{\alpha \uparrow 0}{=} \f{\alpha}{|\Omega|} \int_{\Omega} V(x) \, d^n x + \Oh\big(\alpha^2\big),    \lb{3.7} 
\end{equation}
where $|\Omega|$ abbreviates the volume of $\Omega$. 
\end{corollary}
\begin{proof}
Consider the densely defined closed symmetric operator 
\begin{equation}
 Sf=-\Delta f,\quad f \in \dom (S)= H^2_0(\Omega)=\overline{C_0^\infty(\Omega)}^{\Vert\cdot\Vert_{H^2(\Omega)}},
\end{equation}
in $L^2(\Omega)$ and note that $S$ is semibounded from below by $\kappa>0$, where $\kappa$ is the smallest eigenvalue of the Friedrichs (or Dirichlet) extension
\begin{equation}
 A_F f=-\Delta f,\quad  f \in \dom (A_F)=\big\{g \in H^{3/2}_\Delta(\Omega) \, \big| \, \tau_D g=0\big\};
\end{equation}
cf. \cite[Theorem 6.9 and Lemma 6.11]{BGM19} or \cite[Theorem 2.10 and Lemma 3.4]{GM08} and also \cite[Theorem B.2]{JK95}.
The Neumann Laplacian $A_N$ in \eqref{neumannii} 
is a self-adjoint extension of $S$ and one has $\ker(A_N)=\textrm{lin.span}\{1\}$. One notes that the condition $\ker(A_N)\not\subseteq\ker(V)$ in Theorem~\ref{fieldsmedalsk} 
is satisfied for the multiplication operator $V$ as otherwise the constant function would be in $\ker(V)$, which is only possible if $V=0$.

It remains to show that $V$ is relatively compact with respect to $A_N$ as then Hypothesis~\ref{h1} is satisfied and the statement follows from Theorem~\ref{fieldsmedalsk}.
In order to see that $V$ is relatively compact with respect to $A_N$
we shall use that for $0<\delta<1$ one has 
\begin{equation}\label{embeddi5}
\Vert f\Vert_{L^{2q}(\Omega)}\leq C_q\Vert f\Vert_{H^{3/2-\delta/2}(\Omega)} \, \text{ for } \, q\in \begin{cases} [1,\infty] & \text{if}\,\, n=2, \\ 
[1,n/(n-3+\delta)] & \text{if}\,\, n \in \N, \, n \geq 3, 
\end{cases}
\end{equation}
by \cite[Theorem~8.12.6.I]{Bh12}. 
Let us consider the case $n\geq 3$ first. As $\Omega$ is bounded we have $L^{p_2}(\Omega)\subseteq L^{p_1}(\Omega)$, $1\leq p_1\leq p_2\leq\infty$, and hence under our assumptions
there exists $0<\delta<1$ such that $V\in L^p(\Omega)$, where $p=2n/(3-\delta)$. This yields $V\in L^{2r}(\Omega)$, where $r=n/(3-\delta)$. For $s=n/(n-3+\delta)$
we have $1/r+1/s=1$ and the H\"older inequality together with \eqref{embeddi5} leads to
\begin{equation} 
\Vert V f\Vert_{L^2(\Omega)}\leq\Vert V\Vert_{L^{2r}(\Omega)}\Vert f\Vert_{L^{2s}(\Omega)} \leq C_s\Vert V\Vert_{L^{2r}(\Omega)}\Vert f\Vert_{H^{3/2-\delta/2}(\Omega)},
\end{equation} 
so that 
\begin{equation}\label{vbounded}
V:H^{3/2-\delta/2}(\Omega)\rightarrow L^2(\Omega)
\end{equation}
is bounded. In the case $n=2$ it follows in the same way with $V\in L^{2r}(\Omega)$, $r=1$, and $s=\infty$ that the mapping 
$V$ in \eqref{vbounded} is bounded. 

Next, one observes that $(A_N + I)^{-1}:L^2(\Omega)\rightarrow H^{3/2}(\Omega)$ is bounded; this follows, for instance, from the norm equivalences on $\dom(A_N)$ in \cite[Theorem 6.10]{BGM19}.
As $\Omega$ is bounded it is clear that the embedding $H^{3/2}(\Omega)\hookrightarrow H^{3/2-\delta/2}(\Omega)$ is compact (see, e.g., \cite[Theorem~8.12.6.IV]{Bh12}) and hence 
$(A_N + I)^{-1}:L^2(\Omega)\rightarrow H^{3/2-\delta/2}(\Omega)$ is compact. Together with \eqref{vbounded} we obtain that 
$V(A_N + I)^{-1}:L^2(\Omega)\rightarrow L^2(\Omega)$ is compact, that is,  $V$ is relatively compact with respect to $A_N$.

Finally, \eqref{3.7} is a consequence of analytic first-order Rayleigh--Schr\"odinger perturbation theory (see, e.g., \cite[eq.~(II.2.36), Sect.~VII.3]{Ka80}, \cite[p.~5, Theorems~XII.8, XII.9]{RS78}, \cite[Ch.~II]{Re69}), since $|\Omega|^{-1/2}$ is the normalized eigenfunction corresponding to the simple discrete eigenvalue $0$ of $A_N$. 
\end{proof}

For completeness we also discuss the one-dimensional case for a finite interval $\Omega=(a,b)$. In this context we recall that the self-adjoint Neumann Laplacian in $L^2((a,b))$ is given by
\begin{equation}\label{neumanni}
 A_N f=-f'',\quad f \in \dom (A_N)= \bigl\{g \in H^2((a,b)) \, \big| \,g'(a)=g'(b)=0\bigr\}.
\end{equation}

\begin{corollary}\label{cor2}
Let $(a,b)$ be a finite interval, let $A_N$ be the self-adjoint Neumann Laplacian in $L^2((a,b))$, and assume that $V\not=0$ is a nonnegative function
such that
$V\in L^p((a,b))$ with $p\geq 2$.
Then 
\begin{equation}
 (A_N + \alpha V)f=-f''+\alpha V f,\quad f \in \dom (A_N + \alpha V)=\dom (A_N),
\end{equation}
is self-adjoint in $L^2((a,b))$, 
\begin{equation} 
V(A_N -zI)^{-1} \in \cB_{\infty}\big(L^2((a,b))\big),  \quad z\in \rho(A_N), 
\end{equation} 
and 
\begin{equation} 
\sigma(A_N + \alpha V)\cap(-\infty,0)\not=\emptyset \, \text{ for {\bf any} $\alpha<0$.}
\end{equation} 
Moreover, for $0 < -\alpha$ sufficiently small, the unique eigenvalue $\nu(\alpha) \in (-\infty,0)$ of $A_N + \alpha V$ satisfies 
\begin{equation}
\nu (\alpha) \underset{\alpha \uparrow 0}{=} \f{\alpha}{b-a} \int_a^b V(x) \, d x + \Oh\big(\alpha^2\big).   \lb{3.17} 
\end{equation}
\end{corollary}
\begin{proof}
Consider the densely defined closed symmetric operator 
\begin{equation}
 Sf=-f'',\quad f \in \dom (S)= \bigl\{g \in H^2((a,b)) \, \big| \,g(a)=g(b)=g'(a)=g'(b)=0\bigr\},
\end{equation}
in $L^2((a,b))$ and note that $S$ is semibounded from below by $\kappa=(\pi/(b-a))^2>0$. The Neumann Laplacian $A_N$ in \eqref{neumanni} 
is a self-adjoint extension of $S$ and one has $\ker(A_N)=\textrm{lin.span}\{1\}$. Note that the condition $\ker(A_N)\not\subseteq \ker(V)$ in Theorem~\ref{fieldsmedalsk} 
is satisfied for the multiplication operator $V$ as otherwise the constant function would be in $\ker(V)$, which is only possible if $V=0$.
We claim that $V$ is relatively compact with respect to $A_N$. In fact, using the inequality 
\begin{equation} 
\Vert g\Vert_{L^\infty((a,b))}\leq C\Vert g\Vert_{H^1((a,b))}, \quad g\in H^1((a,b)),
\end{equation} 
one has
\begin{align}
\begin{split} 
\Vert Vg\Vert_{L^2((a,b))}\leq  \Vert V\Vert_{L^2((a,b))} \Vert g\Vert_{L^{\infty}((a,b))} \leq C\Vert V\Vert_{L^2((a,b))} \Vert g\Vert_{H^1((a,b))},&  \\ 
g\in H^1((a,b)),&
\end{split} 
\end{align}
and hence $V:H^1((a,b)) \rightarrow L^2((a,b))$ is bounded. Therefore, as  
$(A_N + I)^{-1}:L^2((a,b))\rightarrow H^2((a,b))$ is bounded and the embedding $H^2((a,b))\hookrightarrow H^1((a,b))$ is compact we see that
$(A_N + I)^{-1}:L^2((a,b))\rightarrow H^1((a,b))$ is compact and thus also $V(A_N + I)^{-1}:L^2((a,b))\rightarrow L^2((a,b))$ is compact. 

Relation \eqref{3.17} is the special one-dimensional case of \eqref{3.7} in Corollary~\ref{cor1}. 
\end{proof}

\begin{remark}\label{saythis}
The observations in Corollaries~\ref{cor1} and \ref{cor2} remain valid for more general classes of self-adjoint Laplacians. More precisely, if $\alpha\in L^\infty(\partial\Omega)$
is real-valued, then the Robin Laplacian
\begin{equation}\label{robini}
 A_\alpha f=-\Delta f,\quad f \in \dom (A_\alpha)= \bigl\{g\in H^{3/2}_\Delta(\Omega)\, \big| \,\tau_N g =\alpha\tau_D g\bigr\},
\end{equation}
is self-adjoint in $L^2(\Omega)$ and if, in addition, $A_\alpha$ is nonnegative and $\ker(A_\alpha)\not=\{0\}$, then  
\begin{equation} 
\sigma(A_\alpha+\alpha V)\cap(-\infty,0)\not=\emptyset \, \text{ for {\bf any} $\alpha<0$}
\end{equation} 
by Theorem~\ref{fieldsmedalsk} under the same integrability assumptions on $V$ as in Corollary~\ref{cor1} if $\ker(A_\alpha)\not\subseteq\ker(V)$ holds. The latter condition is satisfied, for instance, if $V(x)>0$ for a.e. $x\in\Omega$.
Similarly, in the case of a finite interval the Neumann realization $A_N$ of $-d^2/dx^2$ in Corollary~\ref{cor2} can be replaced by any nonnegative self-adjoint realization $A$
of $-d^2/dx^2$ in $L^2((a,b))$ such that $\ker(A)\not=\{0\}$. As $\ker(A)\subseteq\textrm{lin.span}\{1,x\}$ in this case, it is clear that $\ker(A) \not\subseteq\ker(V)$ holds. \hfill $\diamond$
\end{remark}

Next, we consider the case of the Neumann (or Kirchhoff) Laplacian on a compact finite (not necessarily connected) 
graph $\Gamma$, which consists of $e<\infty$ edges (finite intervals) $\mathcal E_n$, $n=1,\dots, e$, and 
$v<\infty$ vertices $\mathcal V_m$, $m=1,\dots,v$. One recalls from \cite{BK13,K24} that the self-adjoint Neumann Laplacian in 
$L^2(\Gamma)=\oplus_{n=1}^e L^2(\mathcal E_n)$ is given by 
\begin{align}\label{neumanngraph}
&  A_N f=(-f_n'')_{n=1}^{e},    \\ 
& f \in \dom (A_N) 
 = \left\{g=(g_n)_{n=1}^{e}\,\bigg|\, 
 \begin{matrix}g_n\in H^2(\mathcal E_n), \,g(x_i)=g(x_j), x_i,x_j\in \mathcal V_m,\\
 \sum_{x_j\in \mathcal V_m}\partial g(x_j)=0, m=1,\dots,v, \end{matrix}
 \right\},    \no
 \end{align}
 and that the multiplicity of $0\in\sigma_p(A_N)$ equals the number of connected components of the metric graph $\Gamma$. 

\begin{corollary}\label{graphi}
Let $\Gamma$ be a compact finite 
graph, let $A_N$ be the self-adjoint Neumann Laplacian in $L^2(\Gamma)$, and assume that $V\not=0$ is a nonnegative function
such that
$V\in L^p(\Gamma)$ with $p\geq 2$.
Then 
\begin{equation}
 (A_N + \alpha V)f=A_N f +\alpha V f,\quad f \in \dom (A_N + \alpha V)=\dom (A_N),
\end{equation}
is self-adjoint in $L^2(\Gamma)$, 
\begin{equation} 
V(A_N -zI)^{-1} \in \cB_{\infty}\big(L^2( \Gamma)\big),  \quad z\in \rho(A_N), 
\end{equation} 
and 
\begin{equation} 
\sigma(A_N + \alpha V)\cap(-\infty,0)\not=\emptyset \, \text{ for {\bf any} $\alpha<0$.}
\end{equation} 
\end{corollary}
The proof of Corollary~\ref{graphi} is similar to that of Corollary~\ref{cor2} and hence is not repeated here.

In the next corollary we consider a perturbed Neumann Laplacian in $L^2((0,\infty))$, where $0\in\sigma_p(A)$ is an embedded eigenvalue. 

\begin{corollary}\label{pncor}
Let
\begin{equation}\label{neumannsingu}
 A_N(q) f=-f''+ qf,\quad f \in \dom (A_N(q))= \bigl\{g \in H^2((0,\infty)) \, \big| \,g'(0)=0\bigr\}, 
\end{equation}
where 
\begin{equation}
q(x)=-\frac{2}{x^2+1}+\frac{8x^2}{(x^2+1)^2} = \frac{6x^2 - 2}{(x^2+1)^2},\quad x\geq 0,
\end{equation}
and assume that $V\not=0$ is a nonnegative function
such that $V\in L^2((0,\infty))$. Then 
\begin{equation}
 (A_N(q) + \alpha V)f=-f''+ qf +  \alpha V f,\quad f \in \dom (A_N(q) + \alpha V)=\dom (A_N(q)),
\end{equation}
is self-adjoint in $L^2((0,\infty))$, 
\begin{align} 
& V(A_N(q) -zI)^{-1} \in \cB_{\infty}\big(L^2((0,\infty))\big),  \quad z\in \rho(A_N(q)),   \\
& \sigma_{ess}(A_N(q) + \alpha V)=\sigma_{ess}(A_N(q))=[0,\infty), 
\end{align} 
and 
\begin{equation} \label{juhu}
\sigma(A_N(q) + \alpha V)\cap(-\infty,0)\not=\emptyset \, \text{ for {\bf any} $\alpha<0$.}
\end{equation} 
\end{corollary}
\begin{proof}
Since $q \in L^{\infty}((0,\infty))$, $A_N(q)$ is self-adjoint in $L^2((0,\infty))$ as the same is true for the unperturbed Neumann operator $A_N f=-f''$, $\dom (A_N)=\dom (A_N(q))$.
It is also clear that $\infty$ is in the limit point case for the differential expression $-(d^2/dx^2) + q(x)$, $x \in [0,\infty)$, and since $q\in L^1((0,\infty))$ it follows from \cite[Proposition~6.13.7]{BHS20} that $\sigma_{ess}(A_N(q))=\sigma_{ess}(A_N)=[0,\infty)$. Alternatively, one can argue that the resolvent difference of the full-line Schr\"odinger operator associated with $-(d^2/dx^2) + q(x)$, $x \in \dR$, in $L^2(\dR)$ and the direct sum of the corresponding two half-line Neumann operators in $L^2((-\infty,0)) \oplus L^2((0,\infty))$ is a rank-one operator and combine this with the fact that $q(x) = q(-x)$, $x \in [0,\infty)$, and the full-line Schr\"odinger operator has essential spectrum equal to $[0,\infty)$ as $\lim_{x \to \pm \infty} q(x) = 0$.
Moreover, it is easy to see that $0$ is a simple  eigenvalue of $A_N(q)$ with corresponding normalized eigenfunction 
\begin{equation}
f_0(x)= \f{2}{\pi^{1/2}}  \frac{1}{x^2+1}, \quad x \in [0,\infty),  \quad \|f_0\|_{L^2((0,\infty))} = 1, 
\end{equation}
and it follows from 
\begin{equation}
A_N(q) = B B^* \geq 0,
\end{equation}
that $A_N(q)$ is nonnegative. Here,
\begin{align}
\begin{split} 
& Bf = f' + \phi f, \quad f \in \dom(B) = H^1_0([0,\infty)),   \\
& B^*g = - g' + \phi g, \quad g \in \dom(B^*) = H^1([0,\infty)),   
\end{split}
\end{align}
where 
\begin{equation}
\phi(x) = f_0'(x)/f_0(x) = - \f{2x}{x^2 + 1}, \quad x \in [0,\infty).
\end{equation}
We also note that the condition $\ker(A_N(q))\not\subseteq \ker(V)$ in Theorem~\ref{fieldsmedalsk} 
is satisfied for the multiplication operator $V$ as otherwise $V=0$.
We claim that $V$ is relatively compact with respect to $A_N(q)$. In fact, for $z\in\dC\setminus [0,\infty)$ we have the identity 
\begin{equation}
(A_N(q)-zI)^{-1}=(A_N-zI)^{-1}-(A_N-zI)^{-1}q(A_N(q)-zI)^{-1}
\end{equation} 
and $V(A_N - zI)^{-1} \in \cB_{\infty}(L^2((0,\infty)))$ by \cite[Problem~41]{RS78} (for the half-line),
and thus also $V(A_N(q) - zI)^{-1} \in \cB_{\infty}(L^2((0,\infty)))$. This implies $\sigma_{ess}(A_N(q) + \alpha V)=\sigma_{ess}(A_N(q))=[0,\infty)$ and hence 
\eqref{juhu} follows from Theorem~\ref{fieldsmedalsk}.
\end{proof}

\begin{remark} \lb{r3.6}
Without going into more details we note that Corollary \ref{pncor} permits the analog of \eqref{3.7} and \eqref{3.17} in the following form: For $0 < -\alpha$ sufficiently small, the unique eigenvalue 
$\nu(\alpha) \in (-\infty,0)$ of $A_N(q) + \alpha V$ satisfies 
\begin{equation}
\nu (\alpha) \underset{\alpha \uparrow 0}{=} 4 \alpha \pi^{-1} \int_0^{\infty} \big(x^2+1\big)^{-2} V(x) \, d x + \Oh\big(\alpha^2\big).     \lb{3.38} 
\end{equation}
While \eqref{3.38} is not a result of analytic first-order Rayleigh--Schr\"odinger perturbation theory as $0$ is not a discrete eigenvalue of $A_N(q)$, one can apply the Fredholm determinant approach developed by Simon \cite{Si76} to arrive at \eqref{3.38}. 
\hfill $\diamond$
\end{remark}

\noindent {\bf Acknowledgments.} 
We are indebted to Petr Siegl for fruitful discussions and helpful remarks. J.B.\ is most grateful for a stimulating research stay at Baylor University, where some parts of this paper were written in October of 2023.~F.G.\ and H.S.\ gratefully acknowledge kind invitations to the Institute for Applied Mathematics at the Graz University of Technology, Austria. 
This research was funded by the Austrian Science Fund (FWF)
Grant-DOI: 10.55776/P33568.
This publication is also based upon work from COST Action CA 18232 MAT-DYN-NET, supported by COST (European Cooperation in Science and Technology), www.cost.eu. 



\end{document}